\documentclass[a4paper]{amsart}
\usepackage{amssymb}
\usepackage[all]{xy}
\SelectTips{cm}{}
\calclayout

\theoremstyle{plain}
\newtheorem{thm}{Theorem}[section]
\newtheorem{cor}[thm]{Corollary}
\newtheorem{lem}[thm]{Lemma}
\newtheorem{exm}{Example}
\newtheorem{prop}[thm]{Proposition}
\newtheorem{defn}[thm]{Definition}
\newtheorem{rem}[thm]{Remark}

\newcommand{\smatrix}[1]{\left(\begin{smallmatrix}#1\end{smallmatrix}\right)}

\def\Inj{\operatorname{Inj}}
\def\inj{\operatorname{inj}}
\def\mod{\operatorname{mod}}
\def\Mod{\operatorname{Mod}}

\def\Hom{\operatorname{Hom}}
\def\End{\operatorname{End}}

\def\Rad{\operatorname{Rad}}
\def\soc{\operatorname{soc}}
\def\rad{\operatorname{rad}}

\def\Ker{\operatorname{Ker}}
\def\Im{\operatorname{Im}}

\def\Coker{\operatorname{Coker}}
\def\supp{\operatorname{supp}}
\def\bdim{\operatorname{\bold{dim}}}
\def\uMod{\operatorname{\underline{Mod}}\nolimits}
\def\umod{\operatorname{\underline{mod}}\nolimits}
\def\mod{\operatorname{mod}}
\def\Mod{\operatorname{Mod}}

\def\Hom{\operatorname{Hom}}
\def\End{\operatorname{End}}
\def\proj{\operatorname{proj}}
\def\inj{\operatorname{inj}}

\def\Inj{\operatorname{Inj}}

\def\Im{\operatorname{Im}}
\def\Ker{\operatorname{Ker}}
\def\Coker{\operatorname{Coker}}

\def\Ab{\operatorname{Ab}}
\def\Zg{\operatorname{Zg}}
\def\fun#1#2#3{#1: #2 \rightarrow #3}
\def\mabb#1{\mathbb{#1}}
\def\mac#1{\mathcal{#1}}

\def\db #1{D^b(\mod #1)}
\def\a{\alpha}
\def\b{\beta}

\def\rta{\rightarrow}
\def\a{\alpha}
\def\b{\beta}
\def\la{\lambda}
\def\La{\Lambda}
\def\Sg{\Sigma}
\def\sg{\sigma}
\def\bu{\bullet}

\def\aif{A^{\infty}_{\infty}}
\title{}
\author{}

\begin{document}
\title{Homotopy categories of injective modules over derived discrete algebras}

\author{Zhe Han}
\address{Zhe Han \\School of Mathematics and Information Science\\
Henan University \\Kaifeng 475001\\ China.}
\address{Fakult\"at f\"ur Mathematik\\
Universit\"at Bielefeld \\D-33501 Bielefeld\\ Germany.}
\email{hanzhe0302@gmail.com}
\begin{abstract}
We study the homotopy category $K(\Inj A)$  of all injective modules over a finite dimensional algebra $A$ with discrete derived category.
We give a classification of the indecomposable objects of $  K(\Inj A)$ for any radical square zero self-injective algebra $A$.
In particular, every indecomposable object is endofinite.
\end{abstract}

\maketitle
\setcounter{tocdepth}{1}

\section{Introduction}
For a finite dimensional algebra $A$ over a algebraically closed field $k$, we investigate the homotopy category $K(\Inj A)$ of all injective $A$-modules.  By \cite{kr05}, we know that $K(\Inj A)$ is compactly generated triangulated category.
The subcategory of compact objects in $K(\Inj A)$ is triangle equivalent to $\db A$, the bounded derived category of finitely generated $A$-modules. Thus we can think of $K(\Inj A)$ as a `compactly generated complete' of $\db A$. The category $K(\Inj A)$ is also very complicated in general. Only in
some special cases, every object in $K(\Inj A)$ can be expressed as direct sum of indecomposable objects.

We focus on the algebra $A$ with discrete derived category $\db A$ \cite{v01} and try to understand the category $K(\Inj A)$.
An algebra $A$ with discrete derived category has a rather simple derived category $\db A$. We expect that $K(\Inj A)$ is easy to control in this case.

For every radical square zero self-injective algebra $A$, it has a discrete derived category. We know that
every basic radical square zero self-injective algebra $A$ is of the form
$kC_n/I_n,n\geq 1$, where the quiver $C_n$ is given by
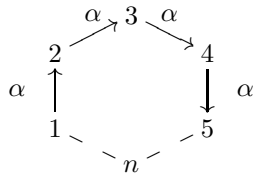
\begin{center}
\begin{tabular}{ccc}
$\xy <1cm,0cm>:
(0,0)*+{1}="1",
(0,1)*+{2}="2",
(1,1.5)*+{3}="3",
(2,1)*+{4}="4",
(2,0)*+{5}="5",
(1,-0.5)*+{n}="6",
(-0.5,0.5)*+{\a}="a1",
(0.5,1.5)*+{\a}="a2",
(1.5,1.5)*+{\a}="a3",
(2.5,0.5)*+{\a}="a4",

\ar@{->}"1";"2"<0pt>
\ar@{->}"2";"3"<0pt>
\ar@{->}"3";"4"<0pt>
\ar@{->}"4";"5"<0pt>
\ar@{--}"5";"6"<0pt>
\ar@{--}"6";"1"<0pt>


\endxy$
\end{tabular}
\end{center}
with relations $I_n=\a^2$.
 Let $\uMod \hat{A}$ be the stable modules category of the repetitive algebra $\hat{A}$ of $A$. We study the category $  K(\Inj A)$ using the fully faithful triangle functor $F:  K(\Inj A)\rta \uMod \hat{A}$
constructed in \cite{kl06}. There is no explicit description of this functor; in particular, we do not know its image. For symmetric algebras we find another embedding which has an explicit expression. By this new embedding, we are able to describe the image of the indecomposable objects in $K(\Inj kC_1/I_1)$ and then extend to general cases.
This leads to a full classifiction of indecomposable objects in $K(\Inj A)$ for radical square zero self-injective algebras $A$.

Our main result is the following.
\begin{thm}
If $A$ is of the form $kC_n/I_n$, then the indecomposable objects of $K(\Inj A)$, up to shifts, are exactly the truncations $\sg_{\leq m}\sg_{\geq l}I^{\bu}$ ,
where $I^{\bullet}$ is the periodic complex

\[\xymatrix{\ldots\ar[r]&I_n\ar[r]^{d_n}&I_{n-1}\ar[r]^{d_{n-1}}&I_{n-2}\ar[r]&\ldots\ar[r]&I_1\ar[r]^{d_1}&I_n\ar[r]^{d_n}&\ldots}\]

with differential $d_j:I_j\rta I_{j-1}$ being composite of $I_j\rta I_j/\soc I_j\cong \rad I_{j-1}\hookrightarrow I_{j-1}$. Moreover, all indecomposable objects are endofinite.
 \end{thm}

The paper is organized as follows. In section 2, we give some background materials about derived discrete algebras, compactly generated triangulated categories and quiver representations. In section 3,
 we give the classification of the indecomposable
objects in the homotopy category $  K(\Inj A)$ for $A$ self-injective and
radical square zero. In section 4 we describe the Ziegler spectrum of $  K(\Inj A)$ for radical square zero self-injective algebras $A$.

\section{Preliminaries}
\subsection{Derived discrete algebras}
Throughout the paper, we fix the field $k$ to be algebraically closed. For an algebra $A$, let $\db A$ be the bounded derived category of $\mod A$. For a complex $X$ in $\db A$,
we define the \emph{homological dimension} of $X$ to be the vector
 $h\bdim (X)=(\bdim H^i(X))_{i\in \mabb Z}$, where $\bdim H^i(X)$ is the dimension vector of $A$-module $H^i(X)$.

There are some algebras with this property that there are only finitely many indecomposable objects in the derived category with the some homological dimension \cite{v01}.

\begin{defn}
 A derived category $\db A$ is discrete if for every vector $d=(d_i)_{i\in \mabb Z}$ with $d_i\geq 0$, there are only finitely many isomorphism classes of indecomposables $X\in\db A$ such that $h\bdim X=d$.
\end{defn}
Following \cite{ba06},
we call the algebra $A$ \emph{derived discrete}.
Let $\Omega$ be the set of all triples $(r,n,m)$ of integers such
that $n\geq r\geq 1$ and $m\geq 0$. For each $(r,n,m)\in\Omega$
consider the quiver $Q(r,n,m)$ of the form
\begin{displaymath}
\xymatrix{&&&&1\ar[r]^{\alpha_1}&\ldots\ar[r]^{\alpha_{n-r-2}}&n-r-1\ar[dr]^{\alpha_{n-r-1}}&\\
-m\ar[r]^{\a_{-m}}&\ldots\ar[r]^{\alpha_{-2}}&-1\ar[r]^{\alpha_{-1}}&0\ar[ru]^{\alpha_0}&&&&n-r\ar[dl]^{\alpha_{n-r}}\\
&&&&n-1\ar[lu]^{\alpha_{n-1}}&\ldots\ar[l]^{\alpha_{n-2}}&n-r+1\ar[l]^{\alpha_{n-r+1}}&}
\end{displaymath}
The ideal $I(r,n,m)$ in the path algebra $kQ(r,n,m)$ is generated by
the paths

$\alpha_0\alpha_{n-1},\alpha_{n-1}\alpha_{n-2},\ldots,\alpha_{n-r+1}\alpha_{n-r}$,
 and let $L(r,n,m)=kQ(r,n,m)/I(r,n,m)$.

It is an easy observation that  $gl.\dim L(r,n,m) =\infty$ for $n=r$.

\begin{prop}\cite[Theorem A]{bgs04}\label{dec1}
Let A be a connected algebra and assume that $A$ is not derived hereditary of
Dynkin type. The following  are equivalent:

(1) $\db A$ is discrete.

(2) $\db A\cong \db {L(r,n,m)}$, for some $(r,n,m)\in
\Omega$.

(3) A is tilting-cotilting equivalent to $L(r,n,m)$, for
$(r,n,m)\in\Omega$.
\end{prop}
From this proposition, we only need to consider the algebras of the form $L(r,n,m)$ when we consider the algebras with discrete derived categories.
\subsection{Compactly generated triangulated categories}
\begin{defn}
Let $\mac T$ be a triangulated category with infinite coproducts. An object T in $\mac T$ is compact if $\Hom_{\mac T}(T,-)$ preserves all coproducts. The category $\mac T$ is called
compactly generated if there is  a set $\mabb S$ of compact objects such that $\Hom_{\mac T}(T,\Sg^iX)=0$ for all $T\in \mabb S$ and $i\in \mabb Z$ implies $X=0$.
\end{defn}
Now we assume that $\mac T$ has infinite coproducts. One can show that the category of all compact objects are closed under shifts and triangles. Thus the full subcategory $\mac T^c$ of compact objects
in $\mac T$ is a triangulated category.

\begin{lem}\cite[Lemma 3.2]{nee96}\label{devis}
 Let $\mac T$ be a compactly generated triangulated category with $\mabb S$ a generating set. If $\mac R$
 is  a full subcategory containing $\mabb S$ and closed under forming  infinite direct sums,then there is a triangulated equivalence $\mac R\cong\mac T$.
\end{lem}
For an algebra $A$, $\Mod A$ denotes the category of all $A$-modules, and $\Inj A$ is the category of all injective $A$-modules.
The unbounded derived category $  D(\Mod A)$ is compactly generated. Its full subcategory of compact objects is $  K^b(\proj A)$, the full subcategory of perfect complexes. The homotopy category
 $  K(\Inj A)$ of $\Inj A$ is compactly generated \cite{kr05}. There is a characterization of the subcategory of compact objects in $K(\Inj A)$.
\begin{prop}\cite[Proposition 2.3]{kr05}
 If $A$ is a finite dimensional $k$-algebra, then the category $  K(\Inj A)$ is compactly generated, and there is a natural triangle equivalence \[  K^c(\Inj A)\cong \db A.\]
\end{prop}

\subsection{Repetitive algebras}
Let $A$ be a finite dimensional basic $k$-algebra. $D=\Hom_k(-,k)$ is the
standard duality on $\mod A$. $DA$ is a $A$-$A$-module via $a',a''\in
A, \varphi\in DA,(a'\varphi a'')(a)=\varphi(a'aa'')$.

\begin{defn}The repetitive algebra $\hat{A}$ of $A$ is defined as following,
the underlying vector space is given by
\[\hat{A}=(\oplus_{i\in\mathbb{Z}}A)\oplus(\oplus_{i\in\mathbb{Z}}DA)\]

denote the element of $\hat{A}$ by $(a_i,\varphi_i)_i$, almost all
$a_i,\varphi_i$ being zero. The multiplication is defined by
\[(a_i,\varphi_i)_i\cdot (b_i,\phi_i)_i=(a_ib_i,a_{i+1}\phi_i+\varphi b_i)_i\]
\end{defn}

An $\hat{A}$-module $M$ is given by $M=(M_i,f'_i)_{i\in \mathbb{Z}}$,
where $M_i$ are A-modules and $f_i:DA\otimes_AM_i\rightarrow M_{i+1}$
such that $f_{i+1}\circ (1\otimes f'_i)=0$.
Given $\hat{A}$-modules $M=(M_i,f'_i)$ and $N=(N_i,g'_i)$, the
morphism $h:M\rightarrow N$ is a sequence $h=(h_i)_{i\in
\mathbb{Z}}$ such that
\begin{displaymath}
\xymatrix{DA\otimes_A M_i\ar[r]^{f'_i}\ar[d]^{1\otimes h_i}&M_{i+1}\ar[d]^{h_{i+1}}\\
 DA\otimes N_i\ar[r]^{g'_i}&N_{i+1}}
 \end{displaymath}commutes. Sometimes, we write $(M_i,f'_i)_{i\in \mathbb{Z}}$ as
\[\cdots M_{-1}\sim^{f_{-1}}M_0\sim^{f_0}M_1\sim\cdots.\]

Let $\Mod \hat{A}$ be the category of all left $\hat{A}$-modules, and $\mod \hat{A}$ be the subcategory of finite dimensional modules. They both are Frobenius categories.
Thus the associated stable categories
$\uMod \hat{A}$ and $\umod \hat{A}$ are triangulated categories. Moreover, $\umod\hat{A}$ is the full subcategory of compact objects in $\uMod\hat{A}$.

Happel introduced the embedding functor $\db A\rta \umod A$ \cite{ha87}, and the functor was extended to $  K(\Inj A)\rta \uMod \hat{A}$ of unbounded complexes in \cite{kl06}.

\begin{prop}\cite[Theorem 7.2]{kl06}\label{inj-ebding}
There is a fully faithful triangle functor $F$ which is the composition of
\[\xymatrix{  K(\Inj A)\ar[r]&  K_{ac}(\Inj\hat{A})\ar[r]^{\sim}&\uMod\hat{A}}\]
extending Happel's functor
\[\xymatrix{D^b(\mod A)\ar[r]^{-\otimes_AA_{\hat {A}}}&D^b(\mod\hat{A})\ar[r]&\umod\hat{A}}.\]
The functor $F$ admits a right adjoint $G$ which is the composition
\[\uMod \hat{A}\xrightarrow{\sim}   K_{ac}(\Inj A)\xrightarrow{\Hom_{\hat{A}}(A,-)}  K(\Inj A).\]
\end{prop}

\subsection{Radical square zero algebras}

In \cite{ars95},  there is a connection between an artin algebra $A$ with $\rad^2(A)=0$ and a hereditary algebra with radical square zero, where $\rad(A)$ is the Jacobison radical of $A$, denoted by $J$.

Now, assume $A'$ is any artin algebra, thus the algebra $A=A'/J'^2$ is a radical square zero algebra. We can associate a hereditary algebra
$A^s=(\begin{smallmatrix}A/J&0\\J&A/J\end{smallmatrix})$ with radical square zero to $A$ \cite{ars95}. The radical of $A^s$ is $\rad(A^s)=(\begin{smallmatrix}0&0\\J&0\end{smallmatrix})$.
Let Q be a quiver with vertex set $Q_0={1,2,\ldots,n}$, the
\emph{separated quiver} $Q^s$ of Q has 2n vertices
${1,\ldots,n,1',\ldots,n'}$ and an arrow $l\rightarrow m'$ for every
arrow $l\rightarrow m$ in Q.
 The ordinary quiver  $Q^s$ of $A^s$ coincides with the separated quiver $(Q_A)^s$ of $Q_A$.

We can study the algebra A via the hereditary algebra $A^s$. For any $M\in \Mod A$, we define an $A^s$-module
$[\begin{smallmatrix}M/JM\\JM\end{smallmatrix}]$. Then there is a functor \[S:\Mod A\rta\Mod A^s,\quad M\mapsto [\begin{smallmatrix}M/JM\\JM\end{smallmatrix}].\] The functor $S$ has some nice properties , which
could be generalized to all modules, see \cite[Chapter X]{ars95} or \cite[Proposition 8.63]{jl89}.
\begin{prop}\label{sepa}
Let A,$A^s$ and $S$ as above, then we have the following.
\begin{enumerate}
\item[(1)] $M $and N in $\Mod A$ are isomorphic if and only if $S(M)$ and $S(N)$ are
isomorphic in $\Mod A^s$.
\item[(2)] $M$ is indecomposable in $\Mod A$ if and only if $S(M)$ is
indecomposable in $\Mod A^s$.
\item[(3)] $M$ is projective in  $\Mod A$ if and only if $S(M)$ is
projective in $\Mod A^s$.
\end{enumerate}
\end{prop}

Moreover, the functor $S$ induces a stable equivalence,

 \[S:\uMod A\rta\uMod A^s.\]

Let $Q$ be any quiver, and $V$ be a representation of $Q$, we define the \emph{radical} $\Rad V$ of $V$ to be the subrepresentation of $V$ with $(\Rad V)_i=\sum_{\a:j\rta i}\Im V_{\a}$, and $\Rad^{n+1}=\Rad(\Rad^n V)$.
From this definition, $V/\Rad V$ is
semisimple. The \emph{Jacobison radical} $\rad V$ of $V$ is the intersection of all maximal subrepresentations of $X$,and $\rad^{n+1}=\rad(\rad^n V)$.

In general, $\rad V$ is a subrepresentation
of $\Rad V$, but if $\Rad^n(V)=0$ for some $n\geq 1$, then $\rad V=\Rad V$. If a bounded quiver $(Q,I)$ does not infinite length path, then $\Rad^mV=0$ from some $m\in \mabb N$.
The representation $V$ is called \emph{radical square zero} if $\Rad^2 V=0$. When we consider radical square zero representations of a quiver $Q$, it is equivalent to consider the module category $\Mod kQ/R_Q^2$,
where $R_Q$ is the ideal generated by all arrows.

There is another version for this correspondence via quiver representations \cite{krar}. There are two functors
\[\fun {T'}{Rep(Q,k)}{Rep(Q^s,k)}\]
and \[\fun T{Rep(Q^s,k)}{Rep(Q,k)}\] which are defined as
follows: Given a representation $X$ of $Q$, let $(T'X)_i=(X/\Rad
X)_i$ and $(T'X)_{i'}=(\Rad X/\Rad^2 X)_i$ for each vertex $i\in
Q_0$. For each arrow $\fun aij$ of $Q$, let $\fun
{(T'X)_{\bar\alpha}}{(T'X)_i}{(T'X)_{j'}}$ be the map which is induced
by $X_\alpha$.  Given a representation $Y$ of $Q^s$, let
$(TY)_i=Y_i\oplus Y_{i'}$ for each vertex $i\in Q_0$. For each arrow
$\alpha\in Q_1$, let $(TY)_{\alpha}=\smatrix{0&0\\
Y_{\alpha}&0}$.

We call a representation $X$  \emph{separated} if $(\Rad X)_i=X_i$
for every sink $i$.

\begin{prop}\cite[Proposition 11.2.2]{krar}\label{rad3}
The functors $T'$ and $T$ induce mutually inverse bijections between
the isomorphism classes of radical square zero representations of
$Q$ and the isomorphism classes of separated representations of
$Q^s$.
\end{prop}

\section{The category $  K(\Inj A)$ of radical square zero self-injective algebra $A$}

We have known that $ K(\Inj A)$ is compactly generated triangulated category with $K^c(\Inj A)\cong \db A$. Actually, $K(\inj A)$ is a derived invarant \cite[Proposition 8.1]{BK08}.
When we consider the homotopy category $K(\Inj A)$ for $A$ derived discrete, we only need to consider the algebras of the form $L(r,n,m)$, by Proposition \ref{dec1}.
In particular, we consider the quiver $C_n$ as following,
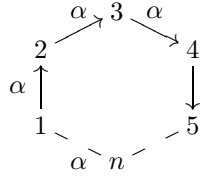
\begin{center}
\begin{tabular}{ccc}
$\xy <1cm,0cm>:
(0,0)*+{1}="1",
(0,1)*+{2}="2",
(1,1.5)*+{3}="3",
(2,1)*+{4}="4",
(2,0)*+{5}="5",
(1,-0.5)*+{n}="6",
(-0.3,0.5)*+{\a}="a1",
(0.5,1.5)*+{\a}="a2",
(1.5,1.5)*+{\a}="a3",
(0.5,-0.5)*+{\a}="a4",

\ar@{->}"1";"2"<0pt>
\ar@{->}"2";"3"<0pt>
\ar@{->}"3";"4"<0pt>
\ar@{->}"4";"5"<0pt>
\ar@{--}"5";"6"<0pt>
\ar@{--}"6";"1"<0pt>


\endxy$
\end{tabular}
\end{center}
with relations $I_n=\a^2$. Set $A_n=kC_n/I_n$.
The algebra of the form $A_n$ is self-injective and $\Rad^2 A_n=0$. There is a characterization of these algebras in \cite[Chapter IV, Proposition 2.16]{ars95}.
\begin{lem}
Let $A$ be a basic self-injective algebra which is not semisimple.
Then $\Rad^2A=0$ if and only if $A\cong A_n$ for some $n$.
\end{lem}

The algebra $A_n=kC_n/I_n$ is derived discrete algebra, since it is the derived discrete algebra $L(n,n,0)$ which occurs in Proposition \ref{dec1}.
Let $I_j, j\in[1,n]$ be the indecomposable injective modules of $A_n$, and $I^{\bullet}$ be the periodic complex with $I_1$ in the degree 0,

\[\xymatrix{\ldots\ar[r]&I_n\ar[r]^{d_n}&I_{n-1}\ar[r]^{d_{n-1}}&I_{n-2}\ar[r]^{d_{n-2}}&\ldots\ar[r]&I_1\ar[r]^{d_1}&I_n\ar[r]&\ldots}\]

where the differential $d_j:I_j\rta I_{j-1}$ is the composition of $I_j\rta I_j/\soc I_j\cong \rad I_{j-1}\hookrightarrow I_{j-1}$.

Define a family of complexes $I_{l,m}=\sg_{\leq m}\sg_{\geq l}I^{\bu}$, where $l\leq m\in \mabb Z\cup \{\pm\infty\}$, where $\sg$ is the brutal truncation functor \cite[p9]{w94}.
In this section, we will prove the following result:
\begin{thm}\label{main1}
 The indecomposable objects in the homotopy category $  K(\Inj A_n)$ are of the form $I_{l,m}[r]$ where $r\in \mabb Z$. Moreover, $I_{l,m}[r]=I_{l',m'}[r']$ if and only if $l'=l+nk, m'=m+nk$ and $r'=r+nk$ where $k\in \mabb Z$.
\end{thm}

\subsection{Classification of indecomposable objects of $  K(\Inj k[x]/(x^2))$}
Given an algebra $A$, the trivial extension of  $A$ is  the algebra $T(A)=A\ltimes D(A)=\{(a,\psi)|a\in A,\phi\in D(A)\}$, where the multiplication is given as following:
\[(a,f)(b,g)=(ab,fb+ag).\]
$T(A)$ is a $\mathbb{Z}$-graded algebra, namely as vector space we have $T(A)=(A,0)\oplus (0,D(A))$, and the elements of $A\oplus 0$ (resp.$0\oplus DA$) is of  degree 0 (resp. degree 1).

The trivial extension $T(A)$ of a finite dimensional algebra $A$ is a symmetric algebra. Since there is a symmetric bilinear pairing $\fun{(-,-)}{T(A)\times T(A)}{T(A)}, ((a,f),(b,g))\mapsto f(b)+g(a)$ satisfying associativity.
\begin{rem}
M is a $\mabb Z$-graded T(A)-module if and only if $M\cong \oplus_{i\in \mabb Z} M_i$, satisfies $M_i\in \Mod A$, and there exists a homomorphism $\fun {f_i}{D(A)\otimes_AM_i}{M_{i+1}}$ for any $i\in\mabb Z$. Denote $\Mod^{\mabb Z}T(A)$ the category of $\mabb Z$-graded $T(A)$-modules.
\end{rem}

Given any $M=\oplus_iM_i\in \Mod^{\mabb Z}T(A)$, there is a module $M=(M_i,f_i)\in \Mod \hat{A}$.
Given any $M=(M_i,f_i)\in \Mod \hat{A}$, it corresponds to the module $M=\oplus_iM_i\in \Mod^{\mabb Z}T(A)$.
This correspondence gives an equivalence of these two categories.
\begin{lem}
 Let $A$ be a finite dimensional k-algebra, them there is an equivalence of categories $\Mod \hat{A}\cong \Mod^{\mabb Z}T(A)$.
\end{lem}

We already know the relations between $\db A$ and $\umod \hat{A}$, which extends the embedding $\mod A\rta \umod\hat{A}$. Now if the algebra $A$ is a symmetric algebra, then we can build some special relations between the complex category of $A$-modules $  C(\Mod  A)$ and $\Mod \hat{A}$.

If $A$ is a symmetric algebra, we have that $D(A)\cong A$ as $A$-$A$-bimodules.
Given a complex in $\Mod A$ \[\xymatrix{\ldots\ar[r]^{d_{i-2}}&X^{i-1}\ar[r]^{d_{i-1}}&X^i\ar[r]^{d_i}&X^{i+1}\ar[r]^{d_{i+1}}&\ldots},\]
we can naturally view the complex as a $\mabb Z$-graded $T(A)$-module $\oplus_{i\in\mabb Z}X_i$ with the morphism $\fun {d_i}{A\otimes_{A}X_i}{X_{i+1}}$. Morphisms of complexes correspond to homomorphisms of $\mabb Z$-graded $T(A)$-modules.
Therefore, we have an embedding functor $S:  C(\Mod A)\rta \Mod^{\mabb Z} T(A), (X^i,d^i)\mapsto (X^i,d^i)_i$ and a forgetful functor $\fun U{\Mod^{\mabb Z} T(A)}{  C(\Mod A)}$.
In this case, the functors $S$ and $U$ are equivalences between $  C(\Mod A)$ and $\Mod^{\mabb Z} T(A)$. By the equivalence $\Mod \hat{A}\cong \Mod^{\mabb Z}T(A)$,
we transform the complexes in $\Mod A$ to the $\hat{A}$-modules.
\begin{lem}
 If $A$ is a symmetric algebra, there is an equivalence of categories $  C(\Mod A)\cong \Mod \hat{A}$.
\end{lem}

For any algebra $A$, the category $\Mod \hat{A}$ is a Frobenius category and the complex category $  C(\Inj A)$ with the set of all degree-wise split exact sequences in $  C(\Inj A)$
is also a Frobenius category.  All indecomposable projective-injective objects in $  C(\Inj A)$ are complexes of the form \[\xymatrix{0\ar[r]&I\ar[r]^{id}&I\ar[r]&0}\] where $I$ is an indecomposable injective
$A$-module.
 The associated stable categories are $\uMod \hat{A}$ and $  K(\Inj A)$ respectively. If $A$ is a symmetric algebra, then $  C(\Inj A)$ is a full exact subcategory of $\Mod \hat{A}$, i.e. $  C(\Inj A)$ is a full subcategory
of $\Mod \hat{A}$ and closed under extensions.

\begin{lem}\label{ebd0}
 Let $A$ be a symmetric algebra, the equivalence $  C(\Mod A)\rta \Mod\hat{A}$ restrict to the embedding $\Psi:  C(\Inj A)\rta \Mod \hat{A}$ is an exact functor. Moreover, the embedding induces a bijection between the indecomposable projective-injective objects
in $  C(\Inj A)$ and $\Mod\hat{A}$.
\end{lem}
\begin{proof}
To show the embedding is an exact functor, it suffices to show $  C(\Inj A)$ is a full exact subcategory of $\Mod \hat{A}$. It is obvious that $  C(\Inj A)$ is closed under extension and the conflations of $\Mod\hat{A}$
with terms in $  C(\Inj A)$ split. Thus the embedding is a fully faithful exact functor.

Now, we consider projective-injective modules in $\Mod \hat{A}\cong \Mod^{\mabb Z} T(A)$. Let $e_1,e_2,\ldots,e_n$
be the primitive idempotents of $A$, and $1=\sum^{n}_{i=0}e_i$ be the unit of $A$.
Let $\mac E_i(e_j)$ be the 'matrix' with $(i,i)$ position is $e_j$, and the other positions are 0.
Then $\{\mac E_i(e_j)\}_{i\in\mabb Z,1\leq j\leq n}$ are all primitive idempotents of $\hat{A}$.

All indecomposable projective-injective $\hat{A}$-modules are of form
$$\hat{P}_i=\hat{A}\mac E_i(e_j)\cong Ae_j\oplus D(A)e_j $$
with $f=id:D(A)\otimes Ae_j\rta D(A)e_j$.
Since $A$ is a symmetric algebra, every indecomposable projective-injective module in $\Mod \hat{A}$ is of the form
\[\cdots 0\sim Ae_j \sim^{1_{Ae_j}}Ae_j\sim0\cdots.\]

The indecomposable projective-injective (associated with the exact structure) objects in $  C(\Inj A)$ are of form
\[\xymatrix{0\ar[r]&Ae_j\ar[r]^{id}&Ae_j\ar[r]&0},\] for some j. Thus there is a natural bijection induced by the embedding $\Psi:  C(\Inj A)\rta \Mod \hat{A}$ between the indecomposable projective-injective modules in $\Mod \hat{A}$
and $  C(\Inj A)$.
\end{proof}
By the lemma, we have that the exact embedding functor $\Psi:  C(\Inj A)\rta   \Mod \hat{A}$ induces an additive functor $\Phi:  K(\Inj A)\rta \uMod \hat{A}$, moreover, $\Phi$ is a triangle functor.

\begin{prop}\label{ebd1}
If $A$ is a symmetric algebra, then there is a fully faithful triangle functor $\Phi:  K(\Inj A)\rta \uMod\hat{A}$ induced by the exact embedding $\Psi:C(\Inj A)\rta \Mod \hat{A}$.
\end{prop}
 \begin{proof}
 Firstly, the embedding preserves projective-injective objects by Lemma \ref{ebd0}. We only need to show that there exists an invertible natural transformation
 $\a:\Psi \Sg\rta \Sg\Psi$ by \cite[Chapter 1, Lemma  2.8]{ha88}. For any object $X\in   C(\Inj A)$, there is a exact sequence $0\rta X\rta I(X)\rta \Sg X\rta 0$. By the fact $\Psi(X)\cong X$ and $\Psi(I(X))\cong I(\Psi(X))$, it is natural that
 $\Psi \Sg X\rta \Sg\Psi(X)$.
 \end{proof}

When we consider the algebra $\La= k[x]/(x^2)$, the embedding $\Psi:C(\Inj \La)\rta\Mod\hat{\La}$ induces a fully faithful triangle functor $\Phi:  K(\Inj\La)\rta\uMod \hat{\La}$, since it is a symmetric algebra.
In the following, we will show that the relation between indecomposable objects of $  K(\Inj \La)$ and radical square zero representations of quiver $\hat{Q}$ of $\hat{\La}$, and determine all indecomposable objects in $  K(\Inj \La)$.

The algebra $\La$ is the path algebra of the quiver $Q=\xymatrix{\circ\ar@(ur,dr)^{\a}}$ with the relation $\a^2=0$.
The quiver of the repetitive algebra $\hat{\La}$ of $\La$ is $\hat{Q}$
\[\xymatrix{\ldots\ar[r]^{\b}&\circ\ar@(lu,ru)^{\a}\ar[r]^{\b}&\circ\ar@(lu,ru)^{\a}\ar[r]^{\b}&\ldots}\]
with relations $\a^2=0=\b^2,\a\b=\b\a$.

\begin{prop}\label{rad2}
Let $\La=k[x]/(x^2)$, as above, $\hat{Q}$ be the quiver of the repetitive algebra $\hat{\La}$.
The image of indecomposable objects in $  K(\Inj \La)$ under $\Phi$ can be expressed as radical square zero representations of $\hat{Q}$.
\end{prop}

\begin{proof}

The objects $X$ in $  K(\Inj \La)$ are of form,\[\xymatrix{\ldots\ar[r]&\La^{m_{-1}}\ar[r]^{d^{-1}}&\La^{m_0}\ar[r]^{d^0}&\La^{m_1}\ar[r]^{d^1}&\La^{m_2}\ar[r]^{d^2}&\ldots}\]

where all $d^i\in\Hom_{\La}(\La^{m_i},\La^{m_{i+1}})$ and satisfy $d^{i+1}\cdot d^i=0$. The differential $d^i$ can be expressed as a $m_{i+1}\times m_i$ matrix $(d^i_{jk})$ with entries in $\Hom_{\La}(\La,\La)$
if $m_{i+1},m_i$ are both finite.
We have that $\dim_k \Hom_{\La}(\La,\La)=2$, so we can choose a basis $\{1,x\}$ of $\Hom_{\La}(\La,\La)$.

In particular, if the complex $X\in   K(\Inj \La)$ is indecomposable, we can choose that every entry $d^i_{jk}$ of the matrix $(d^i_{jk})$ associated to $d^i$ is in $\Rad(\La,\La)$.
Assume that there is a component of morphism $d^i_{jk}:\La\rta\La$ with $d^i_{jk}\notin \Rad(\La,\La)\subset \Hom_{\La}(\La,\La)$. Without loss of generality, let $d^0_{jk}=1_{\La}$.
Consider the following morphisms of complexes in $  K(\Inj A)$
\[\xymatrix{\ldots\ar[r]&\La^{m_{-1}}\ar[d]^0_{g:}\ar[r]^{d^{-1}}&\La^{m_0}\ar[r]^{d^0}\ar[d]^{g_0}&\La^{m_1}\ar[r]^{d^1}\ar[d]^{g_1}&\La^{m_2}\ar[d]^0\ar[r]^{d^2}&\ldots\\
Y:\ldots\ar[r]&0\ar[r]\ar[d]_{f:}^0&\La\ar[d]^{f_0}\ar[r]^1&\La\ar[d]^{f_1}\ar[r]^{d^1}&0\ar[r]\ar[d]^0&\ldots\\
\ldots\ar[r]&\La^{m_{-1}}\ar[r]^{d^{-1}}&\La^{m_0}\ar[r]^{d^0}&\La^{m_1}\ar[r]^{d^1}&\La^{m_2}\ar[r]^{d^2}&\ldots},\]
where $g_0$ is the $k$-th row of $d^0$, $g_1$ is the canonical projection on the $j$-th component, $f_0$ is the embedding to the $k$-th component and $f_1$ is the $j$-th column of $d^0$.
We can check that the morphism $fg:X\rta X$ is idempotent, and $gf=id_Y$. Thus $fg$ splits in $  K(\Inj \La)$ since $  K(\Inj \La)$ is idempotent complete. That means the complex $X$ has a direct summand of form $Y$ which
is null-homotopic.

 We know that $\La$ as a $\La$-module corresponds the quiver representation
$\xymatrix{k^2\ar@(ur,dr)^{\smatrix{0&1\\0&0}}}$, thus we assign the homomorphism $x$ to the morphism of representations
\[\fun {\smatrix{0&1\\0&0}}{\xymatrix{k^2\ar@(lu,ru)^{\smatrix{0&1\\0&0}}}}{\xymatrix{k^2\ar@(lu,ru)^{\smatrix{0&1\\0&0}}}}.\]
Under the embedding functor $\Phi:  K(\Inj \La)\rta \uMod \hat{\La}$ as in Proposition \ref{ebd1},
 the complex $\xymatrix{0\ar[r]&\La\ar[r]^1&\La\ar[r]&0}$ corresponds to the following
representation of $\hat{\La}$
 \[\xymatrix{k^2\ar@(lu,ru)^{\smatrix{0&1\\0&0}}\ar[r]^{id}&k^2\ar@(lu,ru)^{\smatrix{0&1\\0&0}}};\]
  which is a projective-injective $\hat{\La}$-module. Let $\bar{\hat{\La}}$ be the factor algebra
of $\hat{\La}$ modulo its socle \cite{ri97}. The algebra $\bar{\hat{\La}}$ has quiver $\hat{Q}$ and with relations $\a^2=0=\b^2$ and $\a\b=0=\b\a$. Every indecomposable $\hat{\La}$-module without projective summand could be expressed as an indecomposable $\bar{\hat{\La}}$-module.
Thus for any indecomposable complex
 \[X':\xymatrix{\ldots\ar[r]&\La^{m_0}\ar[r]^{d^0}&\La^{m_1}\ar[r]^{d^1}&\La^{m_2}\ar[r]^{d^2}&\ldots}\]  $\Phi(X')\in \uMod \hat{\La}$ could be expressed as an indecomposable
 $\bar{\hat{\La}}$-module.
It naturally corresponds to a radical square zero representation of $\hat{Q}$.
\end{proof}
From the quiver $\hat{Q}$ of $\hat{\La}$, we know that the separated quiver $\hat{Q}^s$ is of type $\aif$ with the following orientation
\[\xymatrix{\ldots\ar[r]&1'&1\ar[l]\ar[r]&2'&2\ar[l]\ar[r]&\ldots}.\]
We denote this quiver by $\aif$.

The representations of $\aif$ are known for experts. For the convenience of readers, we summarize the result in the following proposition.
\begin{prop}\label{reps}
Any indecomposable representation of $\aif$ over an algebraically closed field $k$ is thin.
More precisely, all indecomposable representations are of the form $k_{ab},a,b\in
\mathbb{Z}\cup \{+\infty,-\infty\}$, where
\[k_{ab}(i)=\begin{cases} k &\text{if }a\leq i\leq b,\\ 0&\text{otherwise.}\end{cases}\]
                                                         and
\[k_{ab}(\alpha)=\begin{cases}
id_k &\text{if }a\leq s(\alpha),t(\alpha)\leq b, \\
0 & \text{otherwise.}\end{cases}\]
\end{prop}
\begin{proof}
First, we have that any $k_{ab}$ is indecomposable. If a=b,
obviously $k_{aa}$ is a simple representation and indecomposable,
denoted by $k_a$.

If $a\neq b$, we assume that $a<b$. If $k_{ab}=V\oplus V'$, where
$V$ and $V'$ are nonzero. Then $(suppV)_0\cap (suppV')_0=\emptyset$. There
exists a vertex $i\in (suppV)_0$, and $i+1$ or $i-1\in (suppV')_0$. We have
that $V_{i}=k,V_{i+1}=0,V'(i)=0,V'_{i+1}=k$, and the map $V_{i}\oplus
V'_{i}\rightarrow V_{i+1}\oplus V'_{i+1}$ is zero. But the map $V_{i}\rightarrow
V_{i+1}$ is identity. It is a contradiction.

Second, we need to show any indecomposable representation is of form
$k_{ab}$. Suppose we have a indecomposable representation
$V=(V_i,f_i)$ of $A^{\infty}_{\infty}$. If $|(\supp V)_0|=1$, then V
is the direct sum of simple modules. Since V is indecomposable, we
have that $V\cong k_a$ for some $a\in \mathbb{Z}$.

If $|(\supp V)_0|>1$, then $\supp V$ is connected. Suppose $V$ has
the following form

\[\xymatrix{\ldots\ar[r]&V_{-2}&V_{-1}\ar[l]_{f_{-2}}\ar[r]^{f_{-1}}&V_0&V_1\ar[l]_ {f_0}\ar[r]^{f_1}&V_2&V_3\ar[l]_{f_2}\ar[r]&\ldots}\]

Now if all non-zero $f_i$ are bijections, then $V$  has the form $k_{ab}$
for some $a,b\in \mabb Z\cup \{+\infty,-\infty\}$. Since we have the
following isomorphism of representations,
\[\xymatrix{\ldots\ar[r]&V_{-2}\ar[d]^{f_{-2}\circ f_{-1}}&V_{-1}\ar[d]^{f_{-1}}\ar[l]_{f_{-2}}\ar[r]^{f_{-1}}&V_0\ar[d]^1&V_1\ar[d]^{f_0}
\ar[l]_ {f_0}\ar[r]^{f_1}&V_2\ar[d]^{f_0\circ f^{-1}_{1}}&V_3\ar[d]\ar[l]_{f_2}\ar[r]&\ldots\\
\ldots\ar[r]&V_0&V_0\ar[l]_1\ar[r]^1&V_0&V_0\ar[l]_
1\ar[r]^1&V_0&V_0\ar[l]_1\ar[r]&\ldots}\] Thus  the lower
representation can be decomposed as copies of the corresponding
$k_{ab}$, and $V$ must be isomorphic to $k_{ab}$.

Actually, all non-zero $f_i$ are bijection. If there exists some
$f_i$ which is not an isomorphism. Without loss generality, we
assume that $f_0$ is not injective. In this case, $V_1$ has a
decomposition $V_1=\Im f_0\oplus \Ker f_0$ and $\Ker f_0 \neq0$. We
can choose a basis $\{e_i\}_{i\in I}$ of $\Ker f_0$ and
$\{e'_j\}_{j\in J}$ of $\Im f_0$ such that $\{e_i\}_{i\in I}\bigcup
\{e'_j\}_{j\in J}$ is a basis of $V_1$ and $f_1(\Im f_0)\bigcap
f_1(\Ker f_0)=0$ ($f_1$ is not zero, otherwise there is a non-zero
direct summand $\Ker f_0$ of $V$).

Consider the map $\fun {f_1}{\Im f_0\oplus \Ker f_0}{V_2}$. If $\Ker f_1\neq0$, we get $\Ker f_0\bigcap \Ker f_1=0$,
otherwise the intersection will be a non-zero direct summand of V with support containing only one vertex. We denote $\langle f_1(\Ker f_0)\rangle=V'_2$, the subspace
spanned by the vectors in $f_1(\Ker f_0)$ and the complement
of $V'_2$ in $V_2$ is $V''_2$. The map $f_1$ can be expressed as $$\fun {\smatrix{f'_1&f_{12}\\0&f''_1}}{\Ker f_0\oplus \Im f_0}{V'_2\oplus V''_2}$$. If $f_{12}=0$,
then we already have a decomposition.
If $f_{12}\neq0$, then we can choose a suitable basis of $V_2$ such that $f_{12}=0$. In precisely, If $f_{12}(e'_j)=\sum_{i\in I'}f_1(e_i)\in V'_2$, then replace $e'_j$ by
$e'_j-\sum_{i\in I'}(e_i)$, and we have $f_{12}(e'_j-\sum_{i\in I'}(e_i))=0$. Repeat this procedure, we get a new basis of $\Im f_0$ such that $f_1$
can be expressed as $$\fun {\smatrix{f'_1&0\\0&f''_1}}{\Ker f_0\oplus \Im f_0}{V'_2\oplus V''_2}$$

We can choose the corresponding basis of $V_2$ such that
$V_2\cong V'_2\oplus V''_2$, and
$f_1(\Im f_0)\cong V'_2,\quad f_1(\Ker f_0)\cong V''_2$. Thus we have the following representation
isomorphism
\[\xymatrix{\ldots\ar[r]&V_0\ar[d]^{[\begin{smallmatrix}q_1\\q_2\end{smallmatrix}]}&V_1\ar[d]^{[\begin{smallmatrix}p_1\\p_2\end{smallmatrix}]}\ar[l]_
{f_0}\ar[r]^{f_1}&V_2\ar[d]^1&V_3\ar[d]^1\ar[l]_{f_2}\ar[r]&\ldots\\
\ldots\ar[r]&Imf_0\oplus \Coker f_0&\Im f_0\oplus \Ker f_0\ar[l]_
{\smatrix{f_0&0\\0&0}}\ar[r]^ {\smatrix{f'_1&0\\
0&f''_1}}&V'_2\oplus V''_2&V_3\ar[l]_{f_2}\ar[r]&\ldots}\].

Now we have a nonzero direct summand of V as follows,
\[\xymatrix{\ldots\ar[r]&0\ar[r]&\Ker f_0\ar[r]&V''_2\ar[r]&\ldots}\]
It is contradict with that $V$ is indecomposable.

Use the same argument we can show that if $f_0$ is not surjective, then there also is an nonzero direct summand of $V$.


Given a representation of Q, we can decompose it as direct sum of
indecomposable representations by the procedure and each
indecomposable representation has endomorphism ring $k$. By Krull-Schmidt-Azumaya
Theorem \cite{af74}, this decomposition is unique.
\end{proof}
\begin{cor}\label{rd1}
Let $\La=k[x]/(x^2)$, $Q$ be  the quiver of $\La$,  and $\hat{Q}$ be  the quiver of $\hat{\La}$. Then every indecomposable object in $  K(\Inj\La)$ corresponds to
an indecomposable representation of $\hat{Q}^s$.
\end{cor}
\begin{proof}
By Proposition \ref{rad2}, the functor $\Phi$ sends every indecomposable object in $  K(\Inj \La)$ to a radical square zero representation of $\hat{Q}$.
From Proposition \ref{rad3}, there exists a bijection between radical square zero representations of $\hat{Q}$ and separated representations of $\hat{Q}^s=\aif$.
\end{proof}

We denote by $I_{\La}^{\bu}$ the following acyclic complex in $  K(\Inj \La)$

 \[\xymatrix{\ldots\ar[r] &\La\ar[r]^x&\ldots\ar[r]^x &\La\ar[r]^x&\ldots\ar[r]^x& \La\ar[r]^x&\ldots}.\]
For any $m,n\in \mabb Z\cup \{+\infty,-\infty\}, n\geq m$, we denote the truncation $\sg_{\leq m}\sg_{\geq l}I^{\bu}_{\La}$ by $I^{\La}_{m,n}$.

Now we give the main result in this part.
\begin{prop}\label{dul}
Let $\La=k[x]/(x^2)$, then every indecomposable object in $  K(\Inj \La)$ is of the form $I^{\La}_{m,n}$.
\end{prop}

\begin{proof}
 For an indecomposable object $X\in   K(\Inj \La)$, we consider the embedding functor $\Phi:  K(\Inj \La)\rta \uMod \hat{\La}$ as in Proposition \ref{ebd1}.
We have that $\Phi(X)$ is an indecomposable radical square zero $\hat{A}$-module by Proposition \ref{rad2}. By Corollary \ref{rd1}, every indecomposable object in $  K(\Inj\La)$ corresponds to
an indecomposable representation of $\hat{Q}^s=\aif$. From Proposition \ref{rad3} and \ref{reps}, the dimension of $\Hom_{\hat{\La}}(\hat{\La}e[i], \Phi(X))$ over $k$
 is at most 2 , where $\{e[i]\}_{i\in\mabb Z}$ are all primitive idempotents of $\hat{\La}$. Thus we only have choices $\La$ and $0$ for each $X^i$. $X$ is of
 the form $I^{\La}_{m,n}$, since it is indecomposable.
\end{proof}
\subsection{Proof of the main result}

Covering theory has many applications in representation theory of finite dimensional $k$-algebras. Now, we use the covering technique to classify all indecomposable objects
of $  K(\Inj A)$ for radical square zero self-injective algebras $A$.

Let $\pi:(Q',I')\rta (Q,I)$ be the covering of bounded quiver $(Q,I)$. $\pi$ can be extended to a surjective homomorphism of algebra $\pi:kQ'/I'\rta kQ/I$. We set that $A'=kQ'/I'$ and $A=kQ/I$.
This induces the push down functor $F_{\la}:\Mod A'\rta \Mod A$ \cite{ga81} .

In general, let $G$ be a group, and $A=\oplus_{g\in G}A_g$ be a $G$-graded algebra. We define the covering algebra $\tilde{A}$ associated to the $G$-grading as follows \cite{cm84,ncv04,gh11}:
$\tilde{A}$ is the $G\times G$ matrices $(a_{g,h})$, where $a_{g,h}\in A_{gh^{-1}}$ and all but a finite number of $a_{g,h}$ are 0. Then $\tilde{A}$ is a ring via matrix multiplication and addition.
Set $\mac E=\{e_g\}_{g\in G}$, where $e_g$ is the matrix with 1 in the $(g,g)$-entry and 0 in all other entries.
\begin{prop}\cite{cm84,ncv04}
 Let $G$ be a group and $A=\oplus_{g\in G}A_g$ is a $G$-graded algebra. The covering algebra of $A$ associated to the $G$-grading is $\tilde{A}$. Then $\tilde{A}$ is a locally bounded k-algebra.
Moreover, the category of finitely generated graded $A$-modules $\mod_GA$ is equivalent to $\mod{\tilde{A}}$.
\end{prop}

The forgetful functor $F_{\la}:\mod_GA\rta \mod A$ is the functor sending $X$ to $X$, viewed as an $A$-module. This functor is exactly the pushdown functor $F_{\la}:\mod \tilde{A}\rta \mod A$. The functor $F_{\la}$ is exact
\cite[Proposition 2.7]{gre83}.
By the exactness of functor $F_{\la}$, we have the induced functor $F_{\la}:\db {\tilde{A}}\rta \db A$ between the corresponding derived categories.
\begin{lem}
 Let $G$ be a group and $\tilde{A}$ be the covering algebra associated to a $G$-graded algebra $A$. Then the forgetful functor $F_{\la}:\mod\tilde{A}\rta \mod A$ induces a triangle functor
 $F_{\la}:\db {\tilde{A}}\rta\db A$ such that the following diagram commutes,
\[\xymatrix{\mod \tilde{A}\ar[r]^{can}\ar[d]^{F_{\la}}&\db {\tilde{A}}\ar[d]^{F_{\la}}\\
\mod A\ar[r]^{can}&\db A}.\]
where the functor $can:\mod A\rta \db A$ is the canonical embedding functor.
\end{lem}

Now, assume that the group $G$ is finite, $A$ is a $G$-graded algebra and the induced action of $G$ on $Q_A$ is free , $B$ is the covering algebra associated to the group $G$. Then the covering functor $B\rta A$ induces a covering functor between
the corresponding repetitive algebras $\hat{B}\rta\hat{A}$. The forgetful functor $F_{\la}:\mod \hat{B}\rta \mod \hat{A}$ induces a functor $F_{\la}:\umod \hat{B}\rta\umod\hat{A}$. Consider the full embedding
$i_{\La}:\mod \La\rta \umod\hat{\La}$ for any algebra $\La$, we have $i_A\cdot F_{\la}=F_{\la}\cdot i_B$.

Given an algebra $A$, there is the Happel functor $F^A:\db A\rta \umod\hat{A}$ embedding the  derived category to the stable module category of the repetitive algebra.
For the covering of algebras $B\rta A$ with group $G$, using induction on the length of complexes in $\db B$, we get $F_{\la}\cdot F^B=F^A\cdot F_{\la}$, i.e the following diagram is commutative
\[\xymatrix{\db B\ar[r]^{F^B}\ar[d]^{F_{\la}}&\umod \hat{B}\ar[d]^{F_{\la}}\\
\db A\ar[r]^{F^A}&\umod \hat{A}.}\]

The following result is a consequence of \cite[Lemma 4.5]{bik3} or follows directly from Lemma \ref{devis}.
\begin{lem}\label{brt}
 Let  $F,G:\mac T\rta\mac S$ be two triangle functors preserving coproducts between two $k$-linear compactly generated triangulated categories.
If there is a natural isomorphism $F(X)\cong G(X)$ for any compact object $X\in \mac T^c$,
then $F\cong G$.
\end{lem}

We summarize the above construction. For an algebra $A$, Let $F^A:  K(\Inj A)\hookrightarrow \uMod \hat{A}$ be the functor constructed in Proposition \ref{inj-ebding}
and the functor extends the Happel's functor $F^A:\db A\rta \umod \hat{A}$. If  $\pi:B\rta A$ is the covering of algebra with group $G$ , then there is a covering functor $F_{\la}:\Mod B\rta \Mod A$.
In this case, the functor $F_{\la}$ induces a functor $  K(\Mod B)\rta  K(\Mod A)$ which preserves injective modules. Thus $F_{\la}$ induces a functor denoted by $F_{\la}:  K(\Inj B)\rta   K(\Inj A)$.
\begin{prop}\label{ind-0}
 Let $G$ be a finite group,  $A$ be a $G$-graded algebra and $B$ be the covering algebra of $A$ associated to $G$. Let  the functors $F^A,F^B$ and $F_{\la}$ be as  above.
 Then we have that $F^A\cdot F_{\la}\cong F_{\la}\cdot F^B$, i.e the following commutative diagram
\[\xymatrix{  K(\Inj B)\ar[r]^{F^B}\ar[d]^{F_{\la}}&\uMod \hat{B}\ar[d]^{F_{\la}}\\
  K(\Inj A)\ar[r]^{F^A}&\uMod \hat{A}.}\]
\end{prop}
\begin{proof}
Let $L_1=F^A\cdot F_{\la}$ and $L_2=F_{\la}\cdot F^B$ be two exact functors from $K(\Inj B)$ to $\uMod \hat{A}$.

The given diagram restricted to the subcategory of compact objects is commutative, i.e $L_1(X)\cong L_2(X)$ for all compact object $X\in  K^c(\Inj B)$ by above. We have that
\[\Im L_2|_{K^(\Inj B)}\cong \Im L_1|_{K^c(\Inj B)}\subset \uMod \hat{A}.\]By Lemma \ref{brt}, we only need to show $L_1$ and $L_2$ preserves coproducts. By the properties of functors $F_{\la}$ and $F$, we know that $L_1$ and $L_2$ preserves coproducts.
\end{proof}
Consider the quiver $C_n$ as the beginning of this part. There is a quiver morphism $\fun {\pi}{(C_n,I_n)}{(C_1,I_1)}$ by $\pi(i)=0,i\in [1,n],$ and $\pi(\a)=\a$.
In order to understand $\hat{A}_n$-modules, we need to study the representations of the quiver $\hat{C}_n$.
\begin{prop}
Consider the covering of bounded quivers  $\fun {\pi}{(C_n,I_n)}{(C_1,I_1)}$, let $(\hat{C_n},\hat{I_n})$ and $(\hat{C_1},\hat{I_1})$ be the repetitive quivers of $(C_n,I_n)$ and $(C_1,I_1)$ respectively.
Then $(\hat{C_n},\hat{I_n})\rta (\hat{C_1},\hat{I_1})$  is a covering of  bounded quivers with group $G=\mabb Z_n$.
\end{prop}

The repetitive quiver $\hat{C}_n$ of $C_n$ is given by the following infinite quiver:

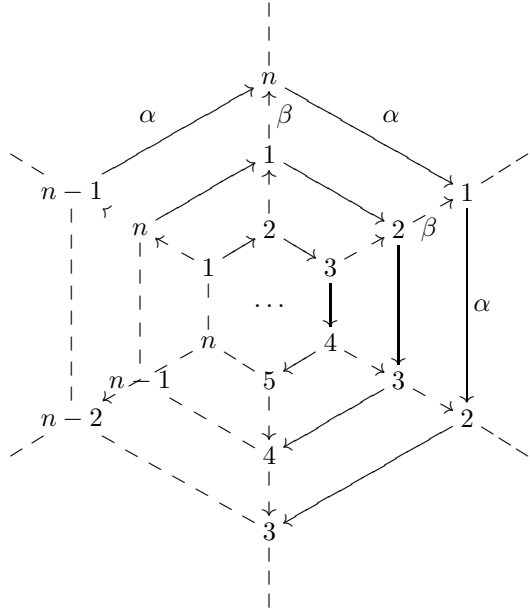
\begin{center}
\begin{tabular}{ccc}
$\xy <1cm,0cm>:
(0,0)*+{\ldots},
(-0.8,0.5)*+{1}="1",
(0,1)*+{2}="2",
(0.8,0.5)*+{3}="3",
(0.8,-0.5)*+{4}="4",
(0,-1)*+{5}="5",
(-0.8,-0.5)*+{n}="6",
(-1.7,1)*+{n}="1a",
(0,2)*+{1}="2a",
(1.7,1)*+{2}="3a",
(1.7,-1)*+{3}="4a",
(0,-2)*+{4}="5a",
(-1.7,-1)*+{n-1}="6a",
(-2.6,1.5)*+{n-1}="1b",
(0,3)*+{n}="2b",
(2.6,1.5)*+{1}="3b",
(2.6,-1.5)*+{2}="4b",
(0,-3)*+{3}="5b",
(-2.6,-1.5)*+{n-2}="6b",
(-1.6,2.5)*+{\a},
(1.6,2.5)*+{\a},
(0.2,2.5)*+{\b},
(2.8,0)*+{\a},
(2.1,1)*+{\b},

\ar@{->}"1";"2"<0pt>
\ar@{->}"2";"3"<0pt>
\ar@{->}"3";"4"<0pt>
\ar@{->}"4";"5"<0pt>
\ar@{--}"5";"6"<0pt>
\ar@{--}"6";"1"<0pt>

\ar@{->}"1a";"2a"<0pt>
\ar@{->}"2a";"3a"<0pt>
\ar@{->}"3a";"4a"<0pt>
\ar@{->}"4a";"5a"<0pt>
\ar@{--}"5a";"6a"<0pt>
\ar@{--}"6a";"1a"<0pt>
\ar@{->}"1b";"2b"<0pt>
\ar@{->}"2b";"3b"<0pt>
\ar@{->}"3b";"4b"<0pt>
\ar@{->}"4b";"5b"<0pt>
\ar@{--}"5b";"6b"<0pt>
\ar@{--}"6b";"1b"<0pt>
\ar@{-->}"1";"1a"<0pt>
\ar@{-->}"2";"2a"<0pt>
\ar@{-->}"3";"3a"<0pt>
\ar@{-->}"4";"4a"<0pt>
\ar@{-->}"5";"5a"<0pt>
\ar@{--}"6";"6b"<0pt>
\ar@{--}"6b";(-3.4,-2)<0pt>
\ar@{--}"1b";(-3.4,2)<0pt>
\ar@{--}"2b";(0,4)<0pt>
\ar@{--}"3b";(3.4,2)<0pt>
\ar@{--}"4b";(3.4,-2)<0pt>
\ar@{--}"5b";(0,-4)<0pt>
\ar@{-->}"1a";"1b"<0pt>
\ar@{-->}"2a";"2b"<0pt>
\ar@{-->}"3a";"3b"<0pt>
\ar@{-->}"4a";"4b"<0pt>
\ar@{-->}"5a";"5b"<0pt>
\ar@{-->}"6a";"6b"<0pt>



\endxy$
\end{tabular}
\end{center}
We label all the original arrows in quiver $C_n$ as $\a$ and all connecting arrows (i.e dash arrows) as $\b$. The relation $\hat{I}_n$ is generated by all the possible $\a^2,\b^2$ and $\a\b=\b\a$. The path algebra of quiver
$(\hat{C}_n,\hat{I}_n)$ is the repetitive algebra $\hat{A}_n$ of $A_n$.

Indecomposable objects in $\uMod\hat{A}_n$ can be expressed as the indecomposable modules of $\bar{\hat{A}}_n$, which is the factor algebra of $\hat{A}_n$ modulo its socle \cite{ri97,er90}. More precisely, the algebra $\bar{\hat{A}}_n$ is given by
the quiver $\hat{C}_n$ and relations $\a^2=0=\b^2$ and $\a\b=\b\a=0$, which is radical square zero.

\begin{prop}\label{ind-1}
 Let $A_n=kC_n/I_n$ be the algebra as above. Every indecomposable object in $  K(\Inj A_n)$ corresponds to an indecomposable module of $\bar{\hat{A}}_n$.
\end{prop}
\begin{proof}
The fully faithful embedding $F:  K(\Inj A_n)\rta \uMod \hat{A}_n$ identifies $  K(\Inj A_n)$ with a localizing subcategory of $\uMod \hat{A}_n$. Thus every indecomposable object can be viewed as
 an indecomposable object in $\uMod\hat{A}_n$ under the functor $F$. It is suffice to consider the indecomposable objects in $\uMod \hat{A}$.
  We know that two $\hat{A}_n$-modules $M,N$, are isomorphic in $\uMod \hat{A}_n$ if and only if there exist projective-injective $\hat{A}_n$-modules $P,Q$ such that
$M\oplus P\cong N\oplus Q$. Furthermore, the indecomposable $\bar{\hat{A}}$-modules are just the non- projective indecomposable $\hat{A}$-modules. It follows that indecomposable objects in $\uMod \hat{A}_n$ corresponds to indecomposable modules of $\bar{\hat{A}}_n$.
\end{proof}
\begin{lem}\label{ind-2}
 Let $F_{\la}:\Mod \hat{A}_n\rta \Mod \hat{A}_1$ be the forgetful functor induced by the covering of bounded quivers  $\fun {\pi}{(C_n,I_n)}{(C_1,I_1)}$. If $X$ is an indecomposable module in $\Mod \bar{\hat{A}}_n$,
then the module
 $Y=F_{\la}X$ is indecomposable in $\Mod\bar{\hat{A}}_1$.
\end{lem}
\begin{proof}
Since $\bar{\hat{A}}_n$ is a radical square zero algebra, therefore there is an bijection between indecomposable modules in $\Mod \bar{\hat{A}}_n$ and indecomposable modules in $\Mod \bar{\hat{A}}^s_n$, where $\bar{\hat{A}}^s_n$ is the
separated algebra of $\bar{\hat{A}}_n$ by Proposition \ref{sepa}. The separated quiver $\hat{C}^s_n$ of $(\hat{C}_n,\bar{\hat{I}}_n)$ is just the union of $n$-copies of quiver $\aif$. Every indecomposable representation of $\hat{C}^s_n$ is an indecomposable representation
of $\aif$, which corresponds to an indecomposable module in $\Mod\bar{\hat{A}}_1$ by Corollary \ref{rd1}.

\end{proof}

\begin{prop}\label{indec}
 The pushdown functor $F_{\la}:  K(\Inj A_n)\rta   K(\Inj A_1)$ preserves indecomposable objects.
\end{prop}

\begin{proof}
 By  Proposition \ref{ind-0},  we have $F_{\la}F^{A_n}\cong F^{A_1}F_{\la}:K(\Inj A_n)\rta \uMod \hat{A}_1.$ Assume that $X\in   K(\Inj A_n)$ is indecomposable and $Y=F_{\la}(X)=Y_1\oplus Y_2$ is a decomposition of $Y\in   K(\Inj A_n)$
with $Y_i\neq 0$ for $i=1,2$. We have that $F_{\la}F^{A_n}(X)\in \uMod\hat{A}_1$ is indecomposable by Proposition \ref{ind-1} and Lemma \ref{ind-2}. On the other hand, $F_{\la}F^{A_n}(X)\cong F^{A_1}F_{\la}(X)=F^{A_1}(Y_1\oplus Y_2)$
is decomposable in $\uMod \hat{A}_1$. This is a contradiction.
\end{proof}

Now, we can give a proof of the main result.
\begin{proof}[ Proof of Theorem \ref{main1}]
If $X\in   K(\Inj A_n)$ with $X^i=\oplus_k I_k$ for some $i\in \mabb Z$, then $F_{\la}(X)\in   K(\Inj A_1)$ with $F_{\la}(X)^i$ being $\oplus_k A_1$. Therefore $F_{\la}(X)$ is decomposable in $  K(\Inj A_1)$
by Proposition \ref{dul}.
 From Proposition \ref{indec}, this implies that $X$ is decomposable in $  K(\Inj A_n)$ .
Thus $X\in  K(\Inj A_n)$ is indecomposable, we have that $F_{\la}(X^i)$ is either $A_1$ or 0 for any $i\in\mabb Z$. It follows that $X^i$ is either $I_k$ or $0$ for some $\leq k\leq n$ and any $i\in\mabb Z$.

 Consider homomorphisms between indecomposable injective $A_n$-modules $I_j$ for $1\leq j\leq n$, we have that
\[\Hom_{A_1}(I_j,I_k)=\begin{cases} (id_{I_j}) &\text{if }j=k;\\
  (d_j) &\text{if }j+1=k; \text{ or } j=n,k=1;
\\0&\text{otherwise},\end{cases}\]
where $d_j:I_j\rta I_{j-1}$ is the composition of $I_j\rta I_j/\soc I_j\cong \rad I_{j-1}\hookrightarrow I_{j-1}$.
From this, the periodic complex $I^{\bullet}$ defined by the following with $I_1$ in the degree 0,

\[\xymatrix{\ldots\ar[r]&I_n\ar[r]^{d_n}&I_{n-1}\ar[r]^{d_{n-1}}&I_{n-2}\ar[r]^{d_{n-2}}&\ldots\ar[r]&I_1\ar[r]^{d_1}&I_n\ar[r]&\ldots}\]
and its shifts $I^{\bu}[r]$ for $r\in \mabb Z$ are exactly the unbounded indecomposable complexes in $K(\Inj A_n)$. Moreover, $I^{\bu}[r]=I^{\bu}[r+n]$ for $r\in \mabb Z$.

Assume that $X\in K(\Inj A)$ is indecomposable and bounded below but not above, i.e there exists $l\in Z$ such that $X^p\neq 0$ for all $p\geq l$ and $X^q=0$ for all $q<l$ . With out loss generality, let $X^l=I_1$, we have
that $X=\sg_{\geq l}I^{\bu}$. In general, $X$ is of the form $\sg_{\geq l}(I^{\bu}[r])$, where $0\leq r\leq n-1$, denoted by $I_{l',+\infty}$. Similarly, if $X\in K(\Inj A_n)$ is indecomposable and bounded above but not below,
then $X=\sg_{\leq m}(I^{\bu}[r])$ where $0\leq r\leq n-1$,denoted by $I_{-\infty,m'}$.

Assume that  $X\in K(\Inj A_n)$ is indecomposable and bounded.  It follows that $X^i\neq 0$ if and only if $l \leq i\leq m$ for some $l,m\in\mabb Z$. Without loss of generality,let $X^l=I-1$, we have that
$X=\sg_{\leq m}\sg_{\geq l}(I^{\bu}[-l])$. In general, $X$ is of the form $I_{l,m}[r]=\sg_{\leq m}\sg_{\geq l}(I^{\bu}[r])$, where $0\leq r\leq n-1$.
\end{proof}

\section{The Ziegler spectrum of $  K(\Inj A)$}
\subsection{Ziegler spectrum of triangulated categories}

Let $\Mod \mac T^c$ be the category consisting of all contravariant additive functor from $\mac T^c$ to $\Ab$, the category of all abelian groups. It is a locally coherent Grothendieck category \cite{he97}.
 Moreover, the objects of form $\Hom(-,X)$ for $X\in\mac T^c$ are projective objects. The subcategory $\mod \mac T^c$ consists of all finitely presented objects of $\Mod \mac T^c$ is abelian category.
The Yoneda embedding
\[h_{\mac T}:\mac T\rta\Mod \mac T^c,X\mapsto H_X=\Hom(-,X)|_{\mac T^c}\]
sends every object $X\in\mac T$ to an object $H_X$ in the abelian category $\Mod \mac T^c$.

The following result characterizes injective objects in the abelian category $\Mod \mac T^c$ \cite[Theorem 1.8]{kra0} .

 \begin{prop}\label{pinj1}Let $\mac T$ be a compactly generated triangulated category, the following are equivalent,
\begin{enumerate}
\item[(1)]  $H_X=\Hom(-,X)|_{\mac T^c}$ is injective in $\Mod \mac T^c$.
\item[(2)] For every set $I$, the summation map $X^{(I)}\rta X$ factors through the canonical map $X^{(I)}\rta X^I$ from the coproduct to the product.
\item[(3)] The map $\Hom(Y,X)\rta \Hom(H_Y,H_X),\phi\mapsto \Hom(-,\phi)|_{\mac T^c}$, is an isomorphism for all $Y\in \mac T$.
\end{enumerate}
\end{prop}
An object $X$ in $\mathcal{T}$ is called \emph{pure-injective} if it satisfies the above equivalent conditions.
There is a class of special pure injective objects, endofinite object, which are analogues of endofinite modules \cite{cb92}.
\begin{defn}
 Let $\mac T$ be a compactly generated triangulated category. An object $E\in \mac T$ is endofinite if the $\End_{\mac T} E$-module $\Hom(X,E)$ has finite length for any $X\in\mac T^c$.
\end{defn}
Endofinite objects of triangulated category have very nice decomposition properties, which can be seen in the following result \cite[Proposition 1.2]{kra99}.
\begin{prop}
 Let $\mac T$ be a compactly generated triangulated category. An endofinite object $X\in \mac T$ has a decomposition $X=\coprod_iX_i$ into indecomposable objects with $\End X_i$ is local, and the decomposition is
unique up to isomorphism.
\end{prop}
An triangulated category $\mac T$ is called \emph{pure semisimple} if every object is pure injective. Analogue of \cite[Theorem 12.20]{bel00}, we have the following result.
\begin{prop}
 For a artin algebra $A$, $  K(\Inj A)$ is pure semisimple if and only if $A$ is derived equivalent to a hereditary algebra of Dynkin type .
\end{prop}
\begin{proof}
The inclusion $i:K^b(\proj A)\hookrightarrow \db A$ induces a functor
\[i^*:\Mod (\db A)\rta \Mod (K^b(\proj A)),\]
which has fully faithful right adjoint functor.
$  K(\Inj A)$ is pure semisimple if and only if $\Mod (\db A)$ is locally Noetherian \cite[Theorem 9.3]{bel00}. By \cite[Chapter 5, Corollary 8.4]{p73}, this implies $\Mod (K^b(\proj A))$ is locally Noetherian. In this case,
 $  D(A)$ is pure semisimple, thus  $A$ is derived equivalence of algebra of Dynkin type.
\end{proof}

The \emph{Ziegler spectrum} $\Zg\mac T $ of  $\mac T$ have its points as the indecomposable injective objects in $\Mod \mac T^c$,
coinciding with the indecomposable pure injective objects in $\mac T$ by Proposition \ref{pinj1}. We will give a topological basis of $\Zg\mac T$, which are determined by the finitely presented objects
 in $\Mod\mac T^c$.

A functor $F:\mac T\rta \Ab$ is \emph{coherent} if there exists an exact sequence $\mac T(X,-)\rta\mac T(Y,-)\rta F\rta 0$, where $X,Y\in\mac T^c$.
We denote $coh \mac T$ the collection of all coherent functors from $\mac T$ to $\Ab$. It is an abelian category. There is an equivalence of categories \cite[Lemma 7.2]{kra002}
\[(\mod \mac T^c)^{op}\xrightarrow{\sim} coh\mac T.\]
We can define the open sets of Ziegler spectrum of $\mac T$ by inducing from the Serre subcategories of $coh\mac T$ or the Serre subcategories of $\mod \mac T^c$. For a Serre subcategory $\mac S$ of $coh\mac T$,
we define the set of the form
\[\mac O(\mac S)=\{X\in\Zg\mac T|C(X)\neq 0, \forall C\in\mac S\}\]
of subsets of $\Zg\mac T$. The sets of this form satisfy the axioms of open subsets in space $\Zg\mac T$.

\begin{lem}\cite{kra002}\label{opb}\begin{enumerate}
  \item The collection of subsets of $\Zg\mac T$
\[\mac O(C)=\{M\in\Zg\mac T|C(M)\neq 0, C\in coh\mac T\}\]
with $C\in coh\mac T$ satisfies the axioms of open subsets of $\Zg\mac T$.
\item The collection of open subsets of $\Zg\mac T$
\[\mac O(C)=\{M\in\Zg\mac T|\Hom(C,H_M)\neq 0, C\in \mod \mac T^c\}\]
with $C\in \mod \mac T^c$ is a basis of open subsets of $\Zg\mac T$.
\end{enumerate}
\end{lem}
\subsection{The Ziegler spectrum of $  K(\Inj A)$}
Let $\mac T$ be a compactly generated triangulated category, for every coherent functor $C\in coh \mac T$, the subset $\mac O(C)=\{M\in\Zg\mac T|C(M)\neq 0\}$ with $C\in coh\mac T$
is open in $\Zg \mac T$.
The corresponding closed subset is $I(C)=\{M\in\Zg\mac T|C(M)= 0\}$. By the equivalence of $(\mod \mac T^c)^{op}=coh\mac T$, the set \[I(C)=\{M\in \Zg\mac T|\Hom(C,H_M)=0\},C\in\mod \mac T^c\] is also a closed set in $\Zg\mac T$.
We apply this to compute the open subsets of some Ziegler spectrum $\Zg\mac T$. The subsets $\mac O(C)$ with $C\in coh\mac T$ form a basis of open subsets of $\Zg\mac T$ by Corollary \ref{opb}.

\begin{prop}
Keep the notions in section 3. Then every indecomposable object $I_{l,m}[r]$of $  K(\Inj A_n)$ is an endofinite object.
 \end{prop}
 \begin{proof}
 We need to calculate the Hom-space $\Hom_{K(\Inj A_n)}(C,I_{l,m}[r])$ for all $C\in   K^c(\Inj A_n)$. For indecomposable objects $I, I'\in   K(\Inj A_n)$,
\[\dim_k\Hom_{  K(\Inj A_n)}(I',I)\leq 1\] and $\Hom_{  K(\Inj A_n)}(I,I)\cong k$. By the fact that $  K^c(\Inj A_n)$ is Krull-Schmidt $k$-linear triangulated category, we know that Hom-space
$\dim_k\Hom_{  K(\Inj A_n)}(C,I_{l,m}[r])$ is finite for all $C\in   K^c(\Inj A_n)$.
 \end{proof}
Every endofinite object in $  K(\Inj A)$ is pure injective, thus the Ziegler spectrum of $  K(\Inj A)$ is explicit.
\begin{cor}
Let $A=kC_n/I_n$ for some $n\in \mabb N^*$. Then the Ziegler spectrum $\Zg(  K(\Inj A))$ consists of the point $[I_{m,n}[r]]$ for each indecomposable object $I_{m,n}[r] \in  K(\Inj A)$ up to isomorphism.
\end{cor}

\begin{exm}

 If $\mac T=  K(\Inj A)$, where $A=k[x]/(x^2)$. It is known that $  K^c(\Inj A)\cong \db A$.
Denote $A_{m,n},A_{-\infty,n},A_{m,+\infty}$ be the complexes with $A$ in degree $m(resp. -\infty,m)$ to $n(resp. n,+\infty)$ and the differential is $d:A\rta A,a\mapsto xa$.
We know that the non zero morphism between complexes in $K^b(\inj A)$ is the linear combinations of the forms:
\begin{enumerate}
 \item[(1)] \[\xymatrix{0\ar[r]&\ldots\ar[r] &A\ar[r]^x&\ldots\ar[r]^x &A\ar[r]^x\ar[d]^0 &\ldots\ar[d]^0\ar[r]^x& A\ar[r]\ar[d]^x &0\\
&&&0\ar[r]&A\ar[r]^x&\ldots\ar[r]^x&A\ar[r]&0}\]
\item[(2)] \[\xymatrix{0\ar[r]&\ldots\ar[r] &A\ar[r]^x\ar[d]^1&\ldots\ar[r]^x &A\ar[r]^x\ar[d]^1&\ldots\ar[r]^x& A\ar[r]&0\\
&0\ar[r]&A\ar[r]^x&\ldots\ar[r]^x&A\ar[r]&0&&}\]
\item[(3)] \[\xymatrix{&&&0\ar[r]&A\ar[r]^x\ar[d]^1&\ldots\ar[r]^x&A\ar[r]\ar[d]^1&0\\
0\ar[r]&\ldots\ar[r] &A\ar[r]^x&\ldots\ar[r]^x &A\ar[r]^x &\ldots\ar[r]^x& A\ar[r]&0}
\]
\item[(4)] \[\xymatrix{&0\ar[r]&A\ar[r]^x\ar[d]^0&\ldots\ar[r]^x&A\ar[r]\ar[d]^x&0&&\\
0\ar[r]&\ldots\ar[r] &A\ar[r]^x&\ldots\ar[r]^x &A\ar[r]^x&\ldots\ar[r]^x& A\ar[r]^x&0}\]
\end{enumerate}
By the description of morphisms, every map $A_{m,n}\rta A_{-\infty,+\infty}$ is 0 in $\mac T$, for any $m,n\in\mabb Z$. Thus $A_{-\infty,+\infty}\notin\mac O(A_{m,n})$.
The functor $H_{A_{m,n}}=\Hom(-,A_{m,n})|_{\mac T^c}\in\mod \mac T^c$, the open subset $\mac O(H_{A_{m,n}})$ consists of the functors represented by complexes $A_{m',n'}\cup \{A[n],A[m]\}$ where $m'<n'$ and $m\leq m'\leq n$ or $m\leq n'\leq m$.
The union of open subsets $\mac O=\cup_{X\in ind (K^b(\inj A))}\mac O(H_X)$ is open. We get that the complement of $\mac O$ has only one point $A_{+\infty,-\infty}$. The point $A_{+\infty,-\infty}$ is not
 an isolated point in $\Zg(\mac T)$ \cite[Proposition 4.5]{gp05}. Thus it is an accumulation point of the indecomposable objects in $K^b(\inj A)$. It is also the accumulation point of the set $\cup_{i\in\mabb Z}\mac O(A[i])$, thus is an accumulation
point of the set $\{A[i]:i\in\mabb Z\}$.

Moreover, the Ziegler spectrum $\Zg\mac T$ is not quasi-compact.
Consider the object $H_{A_{n,+\infty}}\in\mod \mac T^c$, the open subset is of the form
\[\mac O(H_{A_{n,+\infty}})=\{A_{m,+\infty},A_{-\infty,m}(m\geq n),A_{m,r}(m\geq n or r\geq n),A_{-\infty,+\infty}\}.\]
The family of open subsets $\{\mac O(H_{A_{n,+\infty}})\}_{n\in\mabb Z}$ is an open covering of $\Zg\mac T$. But it has no finite subcovering of $\Zg\mac T$.
\end{exm}
\subsection*{Acknowledgements}
This paper is a part of the author's Ph.D thesis. The author would like to thank his supervisor, Henning Krause, for his constant support and inspiring discussions.
He would also like to thank Claus M. Ringel and Dieter Vossieck for their stimulating explanations and viewpoints. The author also thanks Yong Jiang for his suggestions and corrections for
improvements the previous versions of this paper. The author is supported by the China Scholarship Council (CSC).
\bibliographystyle{plain}

\bibliography{bib/thesis}
\end{document}